\documentclass[pdflatex,sn-mathphys-num]{sn-jnl}

\usepackage{amsmath,amssymb,amsfonts}%
\usepackage{algorithm}
\usepackage{algorithmicx}
\usepackage{algpseudocode}%

\graphicspath{ {figures/} }

\theoremstyle{thmstyleone}%
\newtheorem{thm}{Theorem}%

\theoremstyle{thmstyletwo}%

\theoremstyle{thmstylethree}%
\newtheorem{defn}{Definition}%

\theoremstyle{plain}
\newtheorem*{conjecture*}{Conjecture}

\begin{document}

\title[Finite precision block Lanczos]{On finite precision block Lanczos computations}


\author*[1]{\fnm{Dorota} \sur{\v{S}imonov\'a}}\email{simonova@karlin.mff.cuni.cz}
\equalcont{These authors contributed equally to this work.}

\author[2]{\fnm{Petr} \sur{Tich\'{y}}}\email{ptichy@karlin.mff.cuni.cz}
\equalcont{These authors contributed equally to this work.}

\affil[1-2]{\orgdiv{Faculty of Mathematics and Physics}, \orgname{Charles University}, \orgaddress{\street{Sokolovsk\'{a}~83}, \city{Prague}, \postcode{186 65}, \country{Czech Republic}}}

\abstract{In her seminal 1989 work, Greenbaum demonstrated that the results
produced by the finite precision Lanczos algorithm after $k$ iterations
can be interpreted as exact Lanczos results applied to a larger matrix,
whose eigenvalues lie in small intervals around those of the original
matrix. This establishes a mathematical model for finite precision
Lanczos computations. In this paper, we extend these ideas to the
block Lanczos algorithm. We generalize the continuation process and
show that it can be completed in a finite number of iterations using carefully
constructed perturbations. The block tridiagonal matrices produced
after $k$ iterations can then be interpreted as arising from the
exact block Lanczos algorithm applied to a larger model matrix. We
derive sufficient conditions under which the required perturbations
remain small, ensuring that the eigenvalues of the model matrix stay
close to those of the original matrix. While in the single-vector
case these conditions are always satisfiable, as shown by Greenbaum
based on results by Paige, the question of whether they can always
be satisfied in the block case remains open. Finally, we present numerical
experiments demonstrating a practical implementation of the continuation
process and empirically assess the validity of the sufficient conditions
and the size of the perturbations.
}



\maketitle

\bigskip

\section{Introduction}

The Lanczos algorithm is widely used to compute eigenvalue approximations
and to solve linear systems involving a symmetric matrix $A$. Its
behavior has been studied extensively. It can be viewed as the Rayleigh-Ritz
process applied to successive Krylov subspaces. More specifically,
the algorithm computes a sequence of orthogonal restrictions of $A$
onto a sequence of Krylov subspaces with respect to the orthonormal
basis of each subspace. These orthogonal restrictions are represented
by $k\times k$ tridiagonal Jacobi matrices, where $k$ is the iteration
number. The eigenvalues of these matrices, called Ritz values, are
approximations of the eigenvalues of $A$. Alternatively, the Lanczos
algorithm can be viewed as a Stieltjes algorithm for computing orthogonal
polynomials whose roots are equal to the Ritz values. 

The behavior of the Lanczos algorithm can be significantly influenced
by finite precision arithmetic; see, e.g. \cite{MeSt2006}. In particular,
the orthogonality among the computed Lanczos vectors can be lost quickly.
Consequently, clusters of Ritz values that approximate single eigenvalues
may appear. The most extensive analysis of the finite precision Lanczos
algorithm can be found in Paige's doctoral thesis and the series of
papers he published afterward. For example, in~\cite{Pa1980}, the
author analyzed the conditions under which the orthogonality of the
last computed basis vector can be lost. He also showed that the Ritz
values \emph{stabilize} only near the eigenvalues of $A$; here, ``stabilization''
means that, in each subsequent iteration, at least one Ritz value
remains in almost the same position. The approximation properties
of the Lanczos algorithm in finite precision arithmetic were further
studied by Wülling \cite{Wu2006}, building on the earlier work of
Strakoš and Greenbaum \cite{GrSt1992oq}. Among other things, it is
shown that, once a cluster is formed, it closely approximates an eigenvalue
of $A$.

Building on Paige's analysis \cite{Pa1980}, Greenbaum presented a
mathematical model of the finite precision Lanczos algorithm computations
in \cite{Gr1989}. She showed that the \emph{computed} results from
$k$ iterations of the Lanczos algorithm, can be viewed as the results
obtained using the \emph{exact} Lanczos algorithm applied to a larger
matrix with eigenvalues in tiny intervals around the eigenvalues of
$A$. Although the proof is given for intervals of size $\sqrt{\epsilon}\|A\|$,
where $\epsilon$ is the unit roundoff, experiments indicate that
the size can be reduced to $\epsilon\|A\|$. The importance of this
result is supported by an experiment performed in \cite{GrSt1992}
for the conjugate gradient (CG) algorithm, which is closely related
to the Lanczos algorithm. The purpose of the experiment is to compare
the behavior of finite precision CG applied to a given matrix with
that of exact CG applied to a larger matrix whose eigenvalues are
distributed in small intervals around the original matrix's eigenvalues.
The right-hand side of the larger system is constructed from the original
right-hand side so that the sum of the weights corresponding to a
cluster equals the original weight; see \cite{GrSt1992} for more
details on the construction. The authors demonstrate that the behavior
of finite precision CG is numerically very similar to that of exact
CG applied to the blurred system when the intervals are comparable
in size to $\epsilon\|A\|$. These results support the idea that the
Lanczos and CG algorithms are backward stable in the aforementioned
sense (a property we refer to as \emph{backward-like stability}).

The block Lanczos algorithm (see \cite{GoUn1977} or \cite{GoLo2013})
is a version of the Lanczos algorithm that works with block vectors.
The basic properties of the algorithm are summarized, e.g., in Schmelzer's
PhD thesis \cite{Sch2004}. The block Lanczos algorithm takes the
advantage of block operations on modern computer architectures. Moreover,
it builds a richer space which, in theory, could result in faster
convergence of the Ritz values to the eigenvalues. It can also be
used to detect multiple eigenvalues of $A$; however, reorthogonalization
must be used in this case. It seems that without reorthogonalization,
the finite precision block Lanczos algorithm still approximates the
original eigenvalues. Nevertheless, it is then impossible to distinguish
between the approximation of multiple eigenvalues and clusters caused
by rounding errors. 

The finite precision behavior of the block Lanczos algorithm is not
well understood. A few papers briefly discuss this topic, e.g., \cite{GrLeSi1994}
and \cite{XuCh2022}. However, there is no analysis that generalizes
the results of Paige \cite{Pa1980} and Greenbaum \cite{Gr1989}.
To our knowledge, the first attempt at this kind of analysis was started
by Carson and Chen \cite{CaCh}.

The aim of this paper is to study the behavior of the finite precision
block Lanczos algorithm and generalize ideas of Greenbaum about the
mathematical model of finite precision Lanczos computations to the
block case. Generalizing Paige's analysis to the block case is challenging
and requires further research. This paper presents experimental evidence
supporting conjectures and observations that could help with a generalization
approach.

In Section~\ref{sec:The-block-Lanczos}, we introduce the block Lanczos
algorithm and discuss analogies to properties known for the Lanczos
algorithm. To motivate our research on how to mathematically model
the behavior of the block Lanczos algorithm in finite precision arithmetic,
we present a block analogy of Greenbaum and Strakoš's experiment (see
\cite{GrSt1992}) in Section~\ref{sec:Exact-block-CG}. Section~\ref{sec:The-finite-precision}
summarizes what is known about the behavior of the block Lanczos algorithm
in finite precision arithmetic. Section~\ref{sec:FPmodel} presents
the main contribution: a generalization of the construction of Greenbaum's
model of finite precision computations for the block Lanczos algorithm
along with a heuristic strategy for determining the parameters to
obtain a model with desired properties. The final section presents
numerical experiments that support the results of Section~\ref{sec:FPmodel}. 
%The MATLAB codes related to the numerical experiments in Sections \ref{sec:The-finite-precision} and \ref{sec:Experiments} are available on GitHub\footnote{https://github.com/dorotasimonova/On-finite-precision-block-Lanczos-computations.git}.

In this paper, we will refer to the Lanczos algorithm for vectors
as the \textit{single-vector}\emph{ }\textit{\emph{Lanczos algorithm}}.
Unless otherwise stated, all norms are assumed to be 2-norms. $I_{p}$
and $0_{p}$ stands for the identity and the zero matrix, respectively,
of size $p\times p$. Throughout the paper $\epsilon$ is the unit
roundoff.

\section{The block Lanczos algorithm \protect\label{sec:The-block-Lanczos}}

Given a symmetric matrix $A\in\mathbb{R}^{n\times n}$ and a block
vector $v\in\mathbb{R}^{n\times p}$, we can define the $k$th block
Krylov subspace
\[
\mathcal{K}_{k}(A,v)=\mathrm{colspan}\{v,\ldots,A^{k-1}v\},
\]
where ``colspan'' is used to specify the span of individual columns.
Denoting the individual columns of $v$ as $v=[v^{(1)},\ldots,v^{(p)}]$,
it holds that
\[
\mathcal{K}_{k}(A,v)=\mathcal{K}_{k}(A,v^{(1)})+\mathcal{K}_{k}(A,v^{(2)})+\ldots+\mathcal{K}_{k}(A,v^{(p)}).
\]
Therefore, the block Krylov subspace to which we apply the Rayleigh-Ritz
procedure contains more information than the single-vector Krylov
subspaces for each column of~$v$. For simplicity, we will assume
that \textcolor{black}{$\mathrm{dim}\,\mathcal{K}_{k}(A,v)=kp$ for
$k=1,2,\ldots$.}

The block Lanczos algorithm, Algorithm~\ref{alg:bLanczos}, generates
a sequence of orthonormal block vectors $v_{i}\in\mathbb{R}^{n\times p}$,
which means that $v_{i}^{T}v_{j}=I_{p}$ when $i=j$ and $v_{i}^{T}v_{j}=0_{p}$
when $i\neq j$. The block vectors $v_{1},\ldots,v_{k}$ are called
the \textit{block Lanczos vectors} and the columns of these block
vectors form a basis of the corresponding block Krylov subspace. 

\begin{algorithm}[htbp!]
\caption{Block Lanczos \protect\label{alg:bLanczos}}
\begin{algorithmic}[1]

\Require{$A$, $v$}

\State{$v_{0}=0$}

\State{$v_{1}\beta_{1}=v$ } \label{alg:lineQR1}

\For{$k=1,2,\ldots$}

\State{$w=Av_{k}-v_{k-1}\beta_{k}^{T}$}

\State{$\alpha_{k}=v_{k}^{T}w$}

\State{$w=w-v_{k}\alpha_{k}$}

\State{$v_{k+1}\beta_{k+1}=w$} \label{alg:lineQR2}

\EndFor

\end{algorithmic}
\end{algorithm}
On lines \ref{alg:lineQR1} and \ref{alg:lineQR2} of Algorithm~\ref{alg:bLanczos},
the block vector $v_{k+1}$ and the block $\beta_{k+1}\in\mathbb{R}^{p\times p}$
are determined using QR factorization, so that the blocks $\beta_{i}$
are upper triangular matrices. The block vectors and blocks generated
by the block Lanczos algorithm satisfy the relation
\begin{equation}
AV_{k}=V_{k}T_{k}+v_{k+1}\beta_{k+1}e_{k}^{T},\label{eq:bLanczos}
\end{equation}
where $e_{k}^{T}=\left[0_{p},\ldots,0_{p},I_{p}\right]\in\mathbb{R}^{p\times kp}$,
$V_{k}=\left[v_{1},\ldots,v_{k}\right]$ and
\[
T_{k}=\left[\begin{array}{cccc}
\alpha_{1} & \beta_{2}^{T}\\
\beta_{2} & \ddots & \ddots\\
 & \ddots & \ddots & \beta_{k}^{T}\\
 &  & \beta_{k} & \alpha_{k}
\end{array}\right]\in\mathbb{R}^{kp\times kp}
\]
is symmetric and block tridiagonal. 

Multiplying (\ref{alg:bLanczos}) by $V_{k}^{T}$ from the left yields
\[
V_{k}^{T}(V_{k}V_{k}^{T}A)V_{k}=T_{k},
\]
so that $T_{k}$ can be seen as the representing matrix of the orthogonal
restriction of $A$ onto $\mathcal{K}_{k}(A,v)$ with respect to $V_{k}$.
The eigenvalues of $T_{k}$, so-called \textit{Ritz values}, then
approximate the eigenvalues of $A$. Since we assume that the corresponding
block Krylov subspace has full dimension, there are no rank deficiency
problems within the blocks. If $n$ is divisible by $p$, then \textcolor{black}{the
algorithm finishes in the last iteration }$s=\frac{n}{p}$ \textcolor{black}{with
}$\beta_{s+1}=0$, so the Ritz values of $T_{s}$ become a subset
of the eigenvalues of $A$. Note that the general case is more complicated
and may require deflation techniques, such as those based on rank-revealing
QR factorizations.

Let 
\begin{equation}
T_{k}=S_{k}\Theta_{k}S_{k}^{T}\quad\text{with}\quad S_{k}^{T}S_{k}=I_{kp},\label{def:spectral_decomp-T}
\end{equation}
be the spectral decomposition of $T_{k}$, where
\[
S_{k}=\left[s_{1}^{(k)},\ldots,s_{kp}^{(k)}\right]\quad\mbox{and}\quad\Theta_{k}=\mathrm{diag}\left(\theta_{1}^{(k)},\ldots,\theta_{kp}^{(k)}\right).
\]
Multiplying (\ref{eq:bLanczos}) by $S_{k}$ from the right yields
\begin{equation}
AZ_{k}=Z_{k}\Theta_{k}+v_{k+1}\beta_{k+1}\sigma_{p}^{(k)},\label{eq:RitzVecs}
\end{equation}
where $Z_{k}=V_{k}S_{k}$ and $\sigma_{p}^{(k)}=e_{k}^{T}S_{k}$,
\[
Z_{k}=\left[z_{1}^{(k)},\ldots,z_{kp}^{(k)}\right],\quad\sigma_{p}^{(k)}=\left[\sigma_{p,1}^{(k)},\ldots,\sigma_{p,kp}^{(k)}\right].
\]
Taking the $i$th column on both the left and right sides of (\ref{eq:RitzVecs})
we obtain
\begin{equation}
Az_{i}^{(k)}=\theta_{i}^{(k)}z_{i}^{(k)}+v_{k+1}\beta_{k+1}\sigma_{p,i}^{(k)}.\label{eq:RitzVec}
\end{equation}
The vectors $z_{i}^{(k)}=V_{k}s_{i}^{(k)}$, corresponding to the
Ritz values $\theta_{i}^{(k)}$, are called \textit{Ritz vectors}.
The duplets $(\theta_{i}^{(k)},z_{i}^{(k)})$, or Ritz pairs, approximate
the eigenpairs of~$A$. Using the relation (\ref{eq:RitzVec}) we
can estimate the quality of the approximation provided by a given
Ritz pair. It can be easily shown that
\begin{equation}
\min_{j=1,...,k}|\lambda_{j}-\theta_{i}^{(k)}|\leq\frac{\|Az_{i}^{(k)}-\theta_{i}^{(k)}z_{i}^{(k)}\|}{\|z_{i}^{(k)}\|}=\|\beta_{k+1}\sigma_{p,i}^{(k)}\|\equiv\delta_{k,i}.\label{eq:deltaBound}
\end{equation}
The quality of the eigenvalue approximation can therefore be bounded
by $\delta_{k,i}$. 

\subsection{Interlacing}

Using the classical result known from the theory of orthogonal polynomials,
one can shown that the Ritz values from two successive iterations
of the single-vector Lanczos algorithm are strictly interlaced. In
this section we summarize what is known about Ritz values in the block
case. 

Using the general results on eigenvalue interlacing, we can derive
the interlacing principle for two consecutive symmetric block tridiagonal
matrices, $T_{k}$ and $T_{k+1}$, generated by the block Lanczos
algorithm. In particular, considering the spectral decompositions
of $T_{k}$ and $T_{k+1}$ as in (\ref{def:spectral_decomp-T}) and
assuming that $\beta_{j+1}$, $j=1,\ldots,k$, are of full rank, we
can deduce from \cite[p.246]{HoJo2013} that
\begin{equation}
\begin{aligned}\theta_{i}^{(k)}<\theta_{i+p}^{(k+1)}<\theta_{i+p}^{(k)}, & \quad i=1,\ldots,(k-1)p,\\
\theta_{1}^{(k+1)}<\theta_{1}^{(k)}, & \quad\theta_{kp}^{(k)}<\theta_{(k+1)p}^{(k+1)}.
\end{aligned}
\label{eq:interlacingprinc}
\end{equation}

In other words, every open interval formed by $p+1$ consecutive Ritz
values of $T_{k}$ contains at least one Ritz value of $T_{k+1}$.
The assumption of full rank of $\beta$'s is crucial for the strictness
of the inequalities. In the single-vector case, however, the property
is even stronger: between any two consecutive Ritz values from a given
iteration, there is at least one Ritz value from each subsequent iteration.
Considering the last iteration, there is trivially at least one eigenvalue
of $A$ in each interval, see \cite{Sz1975,LiSt2012}. To the best
of our knowledge, there is no result generalizing this property to
symmetric block tridiagonal matrices. The following conjecture proposes
such a generalization.
\begin{conjecture*}
Let $A\in\mathbb{R}^{n\times n}$ be a symmetric matrix, $v\in\mathbb{R}^{n\times p}$
be a block vector. Let $s$ be the largest index such that $\mathcal{K}_{s}(A,v)$
has full dimension. Let $T_{k}$, with spectral decomposition (\ref{def:spectral_decomp-T}),
be the symmetric block tridiagonal matrix generated in the $k$th
iteration of the block Lanczos algorithm applied to $A$ and $v$,
where $0<k<s$. Then each open interval
\[
\begin{aligned}(\theta_{i}^{(k)},\theta_{i+p}^{(k)}), & \quad i=1,\ldots,(k-1)p,\end{aligned}
\]
contains at least one Ritz value of $T_{j}$ for $k<j\leq s$.
\end{conjecture*}
Note that, under the assumptions of the conjecture, the following
two inequalities follow trivially from (\ref{eq:interlacingprinc}):
$\theta_{1}^{(j)}<\theta_{1}^{(k)},$ $\theta_{kp}^{(k)}<\theta_{jp}^{(j)}$. 

We have tested this conjecture numerically on several examples, and
it was confirmed in all cases. 

\subsection{Improper clusters \protect\label{subsec:Clusters-not-approximating}}

For the single-vector Lanczos algorithm, it was shown in \cite{Wu2006}
that if a cluster of Ritz values appears, then it must approximate
an eigenvalue of the original matrix. However, this does not have
to be true for the block Lanczos algorithm, as we will see in this
section.

First, we present a theoretical example, inspired by \cite[p.217]{HnPl2015},
which implies the existence of clusters that do not approximate any
eigenvalue of the original matrix. Suppose we have a sequence of symmetric
tridiagonal matrices 
\[
\widetilde{T}_{k}=\left[\begin{array}{cccc}
\widetilde{\alpha}_{1} & \widetilde{\beta}_{2}\\
\widetilde{\beta}_{2} & \ddots & \ddots\\
 & \ddots & \ddots & \widetilde{\beta}_{k}\\
 &  & \widetilde{\beta}_{k} & \widetilde{\alpha}_{k}
\end{array}\right],\quad k=1,\ldots,
\]
where $\widetilde{\alpha}_{i},\widetilde{\beta}_{i+1}\in\mathbb{R}$
and $\widetilde{\beta}_{i+1}>0$, associated with the single-vector
Lanczos algorithm applied to a symmetric matrix $B\in\mathbb{R}^{s\times s}$
and an initial vector $y\in\mathbb{R}^{s}$. As mentioned above, every
open interval defined by two consecutive Ritz values of $\widetilde{T}_{k}$
contains at least one Ritz value of $\widetilde{T}_{j}$, where $j>k$.
Let $s$ be the smallest index such that $\mathrm{dim}\,\mathcal{K}_{s}(B,y)=\mathrm{dim}\,\mathcal{K}_{s+1}(B,y)$,
and assume that $s\gg p>1$. Define the matrix $T_{s}$ by
\[
T_{s}\equiv\widetilde{T}_{s}\otimes I_{p}=\left[\begin{array}{cccc}
\widetilde{\alpha}_{1}I_{p} & \widetilde{\beta}_{2}I_{p}\\
\widetilde{\beta}_{2}I_{p} & \ddots & \ddots\\
 & \ddots & \ddots & \widetilde{\beta}_{s}I_{p}\\
 &  & \widetilde{\beta}_{s}I_{p} & \widetilde{\alpha}_{s}I_{p}
\end{array}\right],
\]
where $\otimes$ denotes the Kronecker product. Applying the block
Lanczos algorithm to $T_{s}$ and the block vector $e_{1}\otimes I_{p}$,
where $e_{1}\in\mathbb{R}^{s}$ is the first column of the identity
matrix $I_{s}$, yields a sequence of matrices
\[
T_{k}=\widetilde{T}_{k}\otimes I_{p},\quad k=1,\ldots,s.
\]
The matrices $T_{k}$ have the same spectra as $\widetilde{T}_{k}$,
except that each eigenvalue has multiplicity $p$. It also follows
from \cite{HnPl2015} that $T_{k}$ can have eigenvalues of multiplicity
at most~$p$. 
%In the above process, we can use arbitrarily small perturbations
%to produce clusters of Ritz values that do not approximate any eigenvalue
%of~$T_{s}$. 

In the above process, arbitrarily small perturbations of $T_{s}$
can be introduced. When the block Lanczos algorithm is applied to the perturbed 
$T_{s}$, it tends to produce clusters of Ritz values that, 
in general, do not approximate any eigenvalue of the underlying matrix.
For more details on the construction of such a perturbed matrix, see Section~\ref{subsec:ConstructionW}.

We now recall the definition of a cluster and, taking into account
the above considerations, give a name to the new type of cluster.
\begin{defn}
\label{def:clusters} A Ritz value $\theta_{i}^{(k)}$ is said to
be \textit{in a cluster }if
\[
\min_{j\neq i}\frac{|\theta_{i}^{(k)}-\theta_{j}^{(k)}|}{\|A\|}\leq\psi,
\]
for a small $\psi>0$. Otherwise it is said to be \textit{well-separated}
(or just \textit{separated}). A cluster with endpoints $\theta_{min}^{cl},\theta_{max}^{cl}$
increasingly sorted is said to be a \textit{proper cluster} if there
is an eigenvalue $\lambda_{i}$ of $A$ such that
\[
\theta_{min}^{cl}-\eta\|A\|\leq\lambda_{i}\leq\theta_{max}^{cl}+\eta\|A\|,
\]
for a small $\eta>0$. In the other case it is said to be an \textit{improper
cluster}.
\end{defn}

In our experience, improper clusters seem to be related to matrices
with specific eigenvalue distributions. For example, they appear for
the perturbed matrices $\widetilde{T}_{k}\otimes I_{p}$ mentioned above. However,
they rarely appear in experiments with some other matrices. A more
detailed analysis of the appearance of improper clusters is beyond
the scope of this paper, and will be discussed in more detail in \cite{Si}. 

\section{Exact block CG for a blurred problem\protect\label{sec:Exact-block-CG}}

In the single-vector case, a well-known one-to-one correspondence
exists between the Lanczos and CG algorithms. A more complicated analogous
correspondence can also be found between the block Lanczos and block
CG algorithms; see, e.g., \cite{Sa1987,BiFr2014,GuJbSa2004,TiMeSi2025}.
Based on this fact, the goal of this section
is to present an experiment similar to the one in \cite[p.127]{GrSt1992},
which supports the idea of backward-like stability of the Lanczos
and CG algorithms in the single-vector case. We aim to explore whether
an analogous result can be expected in the block case.

As mentioned earlier, Greenbaum showed in \cite{Gr1989} that the
results of finite precision single-vector Lanczos computations can
be viewed as the results of the exact Lanczos algorithm applied to
a larger matrix whose eigenvalues lie within tiny intervals around
the eigenvalues of the original matrix. She presented a particular
way of constructing such a larger matrix. The purpose of the experiment
in \cite{GrSt1992} was to demonstrate that the finite precision CG
behavior mimics the exact CG behavior when applied to a class of matrices
with eigenvalues in tiny intervals around the eigenvalues of $A$.
The right-hand side of the blurred system is constructed from the
original right-hand side vector, as described below.

Our experiment aims to compare the results of finite precision computations
of block CG applied to $Ax=b$ with the results of exact computations
applied to a larger problem, $\hat{A}\hat{x}=\hat{b}$. Let $\lambda_{1},\ldots,\lambda_{n}$
be the eigenvalues of $A$ and $b=\left[b^{(1)},\ldots,b^{(p)}\right]$.
The larger matrix, 
\[
\hat{A}=\text{diag}(\lambda_{1,1},\ldots,\lambda_{1,m},\lambda_{2,1},\ldots,\lambda_{2,m},\ldots,\lambda_{n,1},\ldots,\lambda_{n,m}),
\]
is defined to have $m$ eigenvalues that are uniformly distributed
around each of $A$'s within an interval of width $\delta$, 
\begin{align*}
\lambda_{i,j}=\lambda_{i}+\frac{j-\frac{m+1}{2}}{m-1}\delta, & \quad j=1,\ldots,m.
\end{align*}
The block right-hand side $\hat{b}=[\hat{b}^{(1)},\ldots,\hat{b}^{(p)}]$
of the larger problem is defined using $b$ as follows: for
each column 
\[
\hat{b}^{(i)}=[\hat{b}_{1,1}^{(i)},\ldots,\hat{b}_{1,m}^{(i)},\ldots,\hat{b}_{n,1}^{(i)},\ldots,\hat{b}_{n,m}^{(i)}]^{T}
\]
the elements satisfy 
\[
\hat{b}_{j,1}^{(i)}=\ldots=\hat{b}_{j,m}^{(i)}\quad\text{and}\quad\sum_{t=1}^{m}\left(\hat{b}_{j,t}^{(i)}\right)^{2}=(y_{j}^{T}b^{(i)})^{2},\quad j=1,\ldots,n,
\]
where $Y=\left[y_{1},\ldots,y_{n}\right]$ is the orthonormal matrix
of eigenvectors of $A$. We will compare a quantity analogous to the
relative $A$-norm of error in the single-vector case, which is defined
as
\begin{equation}
\frac{\sqrt{\text{trace}\left((x_{*}-x)^{T}A(x_{*}-x)\right)}}{\sqrt{\text{trace}\left((x_{*}-x_{0})^{T}A(x_{*}-x_{0})\right)}},\label{eq:Anorm}
\end{equation}
where $x_{*}$ is the exact block solution and $x_{0}$ is the initial
guess, which is always the zero block vector in our experiments.

\begin{algorithm}[htbp!]
\caption{O'Leary block CG \protect\label{alg:bCGOLeary}}
\begin{algorithmic}[1]

\Require{$A$, $b$, $x_{0}$}

\State{$r_{0}=b-Ax_{0}$}

\State{$p_{0}=r_{0}\phi_{0}$ }

\For{$k=1,2,\ldots$}

\State{$\gamma_{k-1}=(p_{k-1}^{T}Ap_{k-1})^{-1}\phi_{k-1}^{T}r_{k-1}^{T}r_{k-1}$}

\State{$x_{k}=x_{k-1}+p_{k-1}\gamma_{k-1}$}

\State{$r_{k}=r_{k-1}-Ap_{k-1}\gamma_{k-1}$}

\State{$\delta_{k}=\phi_{k-1}^{-1}(r_{k-1}^{T}r_{k-1})^{-1}r_{k}^{T}r_{k}$}

\State{$p_{k}=(r_{k}+p_{k-1}\delta_{k})\phi_{k}$}

\EndFor

\end{algorithmic}
\end{algorithm}

\begin{algorithm}[th]
\caption{Dubrulle-R block CG \protect\label{alg:BCG_Dubrulle}}

\begin{algorithmic}[1]

\Require{$A$, $b$, $x_{0}$}

\State{$r_{0}=b-Ax_{0}$}

\State{{[}$w_{0},\sigma_{0}]=\mathrm{\mathtt{qr}}(r_{0})$}

\State{$s_{0}=w_{0}$}

\For{$k=1,2,\dots$}

\State{$\xi_{k-1}=\left(s_{k-1}^{T}As_{k-1}\right)^{-1}$}

\State{$x_{k}=x_{k-1}+s_{k-1}\xi_{k-1}\sigma_{k-1}$}

\State{$w=w_{k-1}-As_{k-1}\xi_{k-1}$}

\State{$[w_{k},\zeta_{k}]=\mathrm{\mathtt{qr}}(w)$}

\State{ $s_{k}=w_{k}+s_{k-1}\zeta_{k}^{T}$}

\State{$\sigma_{k}=\zeta_{k}\sigma_{k-1}$}

\EndFor

\end{algorithmic}
\end{algorithm}

The experiment is carried out for two variants of block CG. The first
variant is an algorithm analogous to the Hestenes and Stiefel version
of single-vector CG, as introduced by O'Leary in \cite{OL1980}; see
Algorithm~\ref{alg:bCGOLeary} (HS-BCG). In this algorithm, $\phi_{i}$
is a nonsingular matrix that can be used as a scaling parameter for
the direction vectors. In our experiment, we choose
$\phi_{i}=I_{p}$. The second variant was proposed by Dubrulle in
\cite{Du2001}; see Algorithm~\ref{alg:BCG_Dubrulle} (DR-BCG). This
variant avoids problems with possible rank deficiency within block
vectors. As explained in \textcolor{black}{\cite{TiMeSi2025}}, this
should be the preferred variant of block CG for practical computations.
The exact arithmetic is simulated using double reorthogonalization
of the block vectors $w_{k}$ in DR-BCG. More specifically, the block
vector $w$ is twice orthogonalized againts the previous block vectors
$w_{j}$, $j=0,\dots,k-1$.

For numerical testing, we use the matrix $A=\texttt{bcsstk03}$, a
$112\times112$ matrix from the SuiteSparse Matrix Collection\footnote{https://sparse.tamu.edu},
and $b=\texttt{randn(n,p)}$. The experiment is performed with $p=2$
and $m=11$, using the zero initial guess. First, we apply finite
precision HS-BCG and DR-BCG to $Ax=b$. Then, we apply exact DR-BCG
to larger systems $\hat{A}\hat{x}=\hat{b}$ for two convenient choices
of the parameter~$\delta$.

\begin{figure}
\centering\includegraphics[width=0.7\textwidth]{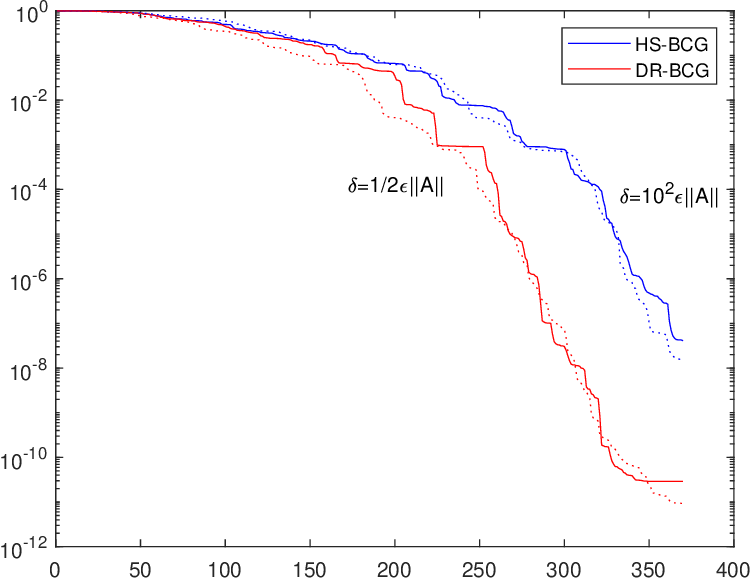}\caption{The quantity (\ref{eq:Anorm}) for finite precision HS-BCG and DR-BCG
applied to $Ax=b$, where $A$ is $\texttt{bcsstk03}$, and for exact
BCG applied to $\hat{A}\hat{x}=\hat{b}$ for $\delta=10^{2}\epsilon\|A\|$
and $\delta=\frac{1}{2}\epsilon\|A\|$. \protect\label{img:experiment}}
\end{figure}

Figure~\ref{img:experiment} shows the quantity (\ref{eq:Anorm})
for finite precision HS-BCG (solid blue) and finite precision DR-BCG
(solid red), both of which are applied to the system $Ax=b$. As expected,
DR-BCG converges faster in finite precision arithmetic because it
avoids rank deficiency problems. Now, we construct the system $\hat{A}\hat{x}=\hat{b}$
with the parameter $\delta=10^{2}\epsilon\|A\|$, apply exact BCG
to it, and plot the quantity (\ref{eq:Anorm}) (blue dotted). Finally,
we construct the system $\hat{A}\hat{x}=\hat{b}$ with the parameter
$\delta=\frac{1}{2}\epsilon\|A\|$ and again apply exact BCG to it
(red dotted). We observe that the convergence curves of the finite-precision
HS-BCG and DR-BCG algorithms closely resemble the exact convergence
curves of BCG applied to their respective model systems. We also performed
the same experiment for $p=3,4$, and $5$. We used slightly different
constants to tune the parameter $\delta$ and obtained very similar
results. Our experiments confirm that the behavior of the finite precision
block CG is similar to that of the exact block CG when applied to
a problem with a matrix whose eigenvalues lie in intervals of size
comparable to $\epsilon\|A\|$ around the eigenvalues of the original
matrix $A$. This experiment motivates our further research and brings
hope that Greenbaum's results on the backward-like stability of the
single-vector Lanczos algorithm may also apply (under some assumptions)
to the block Lanczos algorithm.

\section{The finite precision block Lanczos algorithm \protect\label{sec:The-finite-precision}}

As mentioned previously, we are unaware of any generalization of Paige's
analysis \cite{Pa1980} that would explain the finite precision behavior
of the block Lanczos algorithm. Such a generalization appears to be
non-trivial and is beyond the scope of this paper. This section summarizes
some basic properties of the quantities computed by the block Lanczos
algorithm in finite precision arithmetic. These results will be used
in the next section. 

In the following, the expression 
\[
\mathcal{O}(z),\quad z>0,
\]
refers to an unspecified number whose size can be bounded by $z$
and a constant that may depend on small powers of $n$ (the size of
the problem), $p$ (the width of the block vectors), and $k$ (the
iteration number). 

The computed block Lanczos vectors satisfy a perturbed recurrence
relation (\ref{eq:bLanczos}), which can be written as
\begin{equation}
AV_{k}=V_{k}T_{k}+v_{k+1}\beta_{k+1}e_{k}^{T}+\Delta V_{k},\label{eq:perturbedBLanczos}
\end{equation}
where $\Delta V_{k}=\left[\Delta v_{1},\ldots,\Delta v_{k}\right]\in\mathbb{R}^{n\times kp}$
represents the perturbations due to computations in finite precision
arithmetic. Analogously to the single-vector case, it can be shown
that the size of the perturbations is bounded by 
\begin{equation}
\|\Delta v_{j}\|\leq\mathcal{O}(\epsilon)\|A\|,\quad j=1,\ldots,k.\label{eq:FPperturbations}
\end{equation}
Furthermore, when using Householder QR factorization on lines \ref{alg:lineQR1}
and \ref{alg:lineQR2} of Algorithm~\ref{alg:bLanczos}, the vectors
within a block Lanczos vector are almost exactly orthonormal
\begin{equation}
\|v_{j+1}^{T}v_{j+1}-I_{p}\|\leq\mathcal{O}(\epsilon),\label{eq:FPproperties-normbLanczvecs}
\end{equation}
 and the local orthogonality is also well preserved, i.e.
\begin{eqnarray}
\|v_{j}^{T}v_{j+1}\beta_{j+1}\| & \leq & \mathcal{O}(\epsilon)\|A\|,\quad j=0,\ldots,k.\label{eq:FPproperties-localOG}
\end{eqnarray}
 Finally, it can be shown that 

\begin{equation}
\|\beta_{k+1}\|\leq\mathcal{O}(1)\|A\|.\label{eq:FPbeta}
\end{equation}
The bounds (\ref{eq:FPperturbations}), (\ref{eq:FPproperties-localOG})
and (\ref{eq:FPbeta}) 
have been shown by Carson and Chen \cite{CaCh} and
%\textcolor{magenta}{are based on \cite{CaCh} and }
will be explained in more detail in \cite{Si}. To demonstrate
the validity of the bounds, at least numerically, we present an experiment
in which we plot the actual sizes of the norms from (\ref{eq:FPperturbations})
and (\ref{eq:FPproperties-localOG}).

We consider the matrices introduced in \cite{St1991}, with the eigenvalues
\[
\lambda_{i}=\lambda_{1}+\frac{i-1}{n-1}(\lambda_{n}-\lambda_{1})\rho^{n-i},\quad i=2,\ldots,n,
\]
where $\rho\in(0,1)$ is a density parameter. The matrix $A$ is then
set to $U\Lambda U^{T}$, where $U$ is a random orthonormal matrix
and $\Lambda=\text{diag}(\lambda_{1},\ldots,\lambda_{n})$. In all
experiments, we use the parameters $n=48$ and $\rho=0.8$. For clarity,
we refer to the matrix with $\lambda_{1}=0.1$ and $\lambda_{n}=100$
as $\texttt{strakos48(0.1,100)}$, and the one with $\lambda_{1}=0.001$
and $\lambda_{n}=1$ as $\texttt{strakos48(0.001,1)}$. We choose
$p=2$ and set $v=\texttt{randn(n,p)}$. Finally, we apply the block
Lanczos algorithm to $A$ and $v$. 

\begin{figure}
\includegraphics[width=0.48\textwidth]{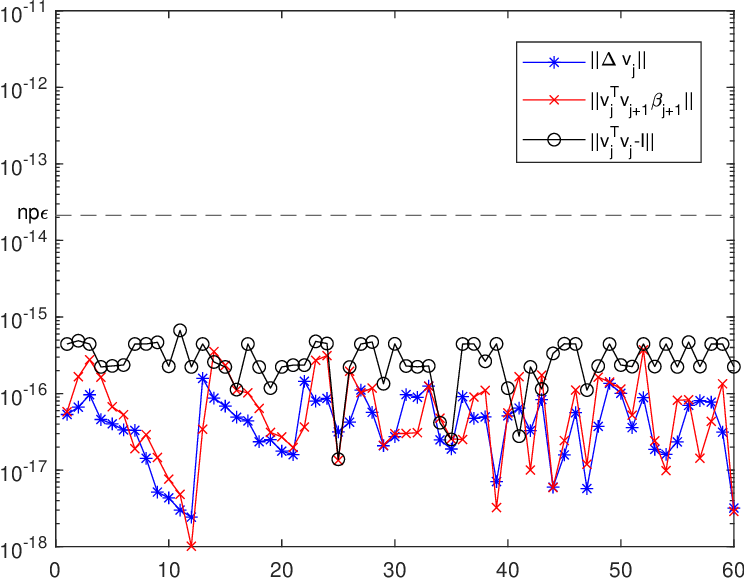}\includegraphics[width=0.48\textwidth]{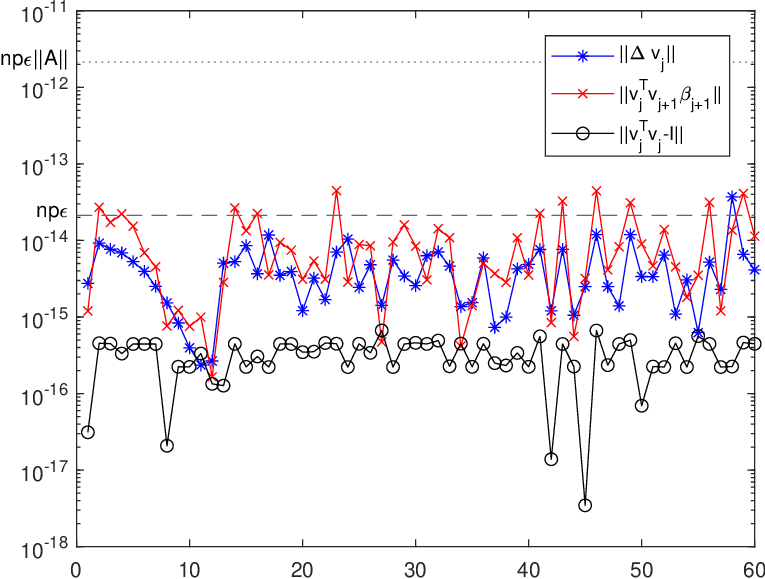}\caption{The terms in (\ref{eq:FPperturbations}), (\ref{eq:FPproperties-normbLanczvecs}),
and (\ref{eq:FPproperties-localOG}) for $A=\texttt{strakos48(0.001,1)}$
(left) and $A=\texttt{strakos48(0.1,100)}$ (right). \protect\label{img:experiment2}}
\end{figure}

In Figure~\ref{img:experiment2} we plot the terms $\|\Delta v_{j}\|$,
$\|v_{j}^{T}v_{j}-I_{p}\|$, and $\|v_{j}^{T}v_{j+1}\beta_{j+1}\|$,
which appear on the left-hand sides of inequalities (\ref{eq:FPperturbations}),
(\ref{eq:FPproperties-normbLanczvecs}), and (\ref{eq:FPproperties-localOG}),
respectively. For a better overview and comparison, we also plot the
level $np\epsilon$ (dashed line) and $np\epsilon\|A\|$ (dotted line),
which represent the approximate sizes of the terms $\mathcal{O}(\epsilon)$
and $\mathcal{O}(\epsilon)\|A\|$. We obtained similar results also
for $p=3,4,5$.

Although the perturbations $\Delta v_{j}$ are small (comparable to
$\epsilon$), they can significantly impact the behavior of the block
Lanczos algorithm in finite precision arithmetic. The orthogonality
among the block Lanczos vectors can be completely lost after a few
iterations. This means that we can no longer consider $V_{k}$ to
be a matrix with (almost) orthonormal columns. 

We will now derive an analogous bound to (\ref{eq:deltaBound}) which
holds for numerically computed quantities. Let $\lambda_{i}$ be the
eigenvalues of $A$ and assume that (\ref{def:spectral_decomp-T})
is the (exact) spectral decomposition of the computed Jacobi matrix
$T_{k}$. Multiplying (\ref{eq:perturbedBLanczos}) by $S_{k}$ and
taking the $i$th column of the resulting block vectors yields
\[
Az_{i}^{(k)}=\theta_{i}^{(k)}z_{i}^{(k)}+v_{k+1}\beta_{k+1}\sigma_{p,i}^{(k)}+\Delta V_{k}s_{i}^{(k)}.
\]
Taking the norm on both sides and using (\ref{eq:FPperturbations})
and (\ref{eq:FPproperties-localOG}), we obtain 
\begin{equation}
\min_{j=1,...,k}|\lambda_{j}-\theta_{i}^{(k)}|\leq\frac{\|Az_{i}^{(k)}-\theta_{i}^{(k)}z_{i}^{(k)}\|}{\|z_{i}^{(k)}\|}\le\frac{\delta_{k,i}\left(1+\mathcal{O}(\epsilon)\right)+\mathcal{O}(\epsilon)\|A\|}{\|z_{i}^{(k)}\|}.\label{eq:deltaFPBound}
\end{equation}
As in the single-vector case, this bound is useful when $\|z_{i}^{(k)}\|$
is not small.

\section{Model of finite precision computations \protect\label{sec:FPmodel}}

In this section, we present a generalization of Greenbaum's construction,
see \cite{Gr1989}, of the model of single-vector Lanczos computations.

Suppose that we have performed $k$ iterations of the finite precision
block Lanczos algorithm applied to $A$ and an initial vector $v$,
using the perturbed recurrences (\ref{eq:perturbedBLanczos}), so
that $V_{k}$, $T_{k}$ and the perturbations $\Delta V_{k-1}=\left[\Delta v_{1},\ldots,\Delta v_{k-1}\right]$
are known. Next, assume that we can determine the perturbations $\Delta\widetilde{V}_{N-k+1}=\left[\Delta\widetilde{v}_{k},\ldots,\Delta\widetilde{v}_{N}\right]$
such that
\begin{equation}
AV_{N}=V_{N}T_{N}+\left[\Delta V_{k-1},\Delta\widetilde{V}_{N-k+1}\right],\label{eq:Grmodel}
\end{equation}
where 
\[
T_{N}=\left[\begin{array}{cccc}
T_{k} & \beta_{k+1}^{T}\\
\beta_{k+1} & \alpha_{k+1} & \ddots\\
 & \ddots & \ddots & \beta_{N}^{T}\\
 &  & \beta_{N} & \alpha_{N}
\end{array}\right]
\]
and $\beta_{N+1}=0$. A method for obtaining (\ref{eq:Grmodel})
will be discussed later in Section~\ref{subsec:The-continuation-process},
along with the sizes of $\Delta\widetilde{V}_{N-k+1}$
and $T_{N}$. First, we need to find a substitute for Paige's theorem;
see \cite[p.~22]{Gr1989}. The following theorem establishes a connection
between the size of the perturbations $[\Delta V_{k-1},\Delta\widetilde{V}_{N-k+1}]$
and the spread of the eigenvalues of $T_{N}$ around the eigenvalues
of~$A$. Since a comprehensive theory describing the finite precision
behavior of the block Lanczos algorithm is still lacking, the theorem
is stated under additional assumptions.
\begin{thm}
\label{thm:size_intervals} Let $T_{N}$ be the block tridiagonal
matrix generated by a perturbed block Lanczos recurrence (\ref{eq:Grmodel})
with perturbations $[\Delta V_{k-1},\Delta\widetilde{V}_{N-k+1}]$,
which satisfy
\[
\max\left(\|\Delta v_{1}\|,\ldots,\|\Delta v_{k-1}\|,\|\Delta\widetilde{v}_{k}\|,\ldots,\|\Delta\widetilde{v}_{N}\|\right)\leq\epsilon_{2}\|A\|
\]
for some $\epsilon_{2}>0$. Furthermore, let $\epsilon_{1}>0$ (independent
of the indices $i$ and $j$) be such that for every Ritz pair $(\theta_{i}^{(N)},z_{i}^{(N)})$
of $T_{N}$ with $\|z_{i}^{(N)}\|<0.5$ there exists a Ritz pair $(\theta_{j}^{(N)},z_{j}^{(N)})$
for which 
\begin{equation}
\|z_{j}^{(N)}\|\geq0.5\quad\text{and}\quad|\theta_{i}^{(N)}-\theta_{j}^{(N)}|\leq\epsilon_{1}\|A\|.\label{eq:assumtion}
\end{equation}
Then each eigenvalue of $T_{N}$ lies within 
\begin{equation}
3\max\left(\sqrt{N}\epsilon_{2},\epsilon_{1}\right)\|A\|\label{eq:size}
\end{equation}
 of an eigenvalue of $A$.
\end{thm}

\begin{proof}
Let $T_{N}=S_{N}\Theta_{N}S_{N}^{T}$ be the spectral decomposition
defined as (\ref{def:spectral_decomp-T}). Multiplying (\ref{eq:Grmodel})
by $S_{N}$ from the right and taking the $i$th column of the resulting
matrix equation yields
\begin{equation}
Az_{i}^{(N)}-\theta_{i}^{(N)}z_{i}^{(N)}=\left[\Delta V_{k-1},\Delta\widetilde{V}_{N-k+1}\right]s_{i}^{(N)}.\label{eq:FPblock-RitzPairs}
\end{equation}
Using (\ref{eq:FPblock-RitzPairs}) and performing some simple algebraic
manipulations, we obtain the following bound:
\begin{align}
\min_{l}|\lambda_{l}-\theta_{i}^{(N)}|\|z_{i}^{(N)}\| & \leq\|Az_{i}^{(N)}-\theta_{i}^{(N)}z_{i}^{(N)}\|\nonumber \\
 & =\left\Vert \left[\Delta V_{k-1},\Delta\widetilde{V}_{N-k+1}\right]s_{i}^{(N)}\right\Vert \leq\sqrt{N}\epsilon_{2}\|A\|.\label{eq:proof_bound}
\end{align}
If $\|z_{i}^{(N)}\|\geq0.5$, then from (\ref{eq:proof_bound}), we
obtain
\[
\min_{l}|\lambda_{l}-\theta_{i}^{(N)}|\leq2\sqrt{N}\epsilon_{2}\|A\|.
\]
On the other hand, if $\|z_{i}^{(N)}\|<0.5$, then the assumptions
of Theorem~\ref{thm:size_intervals} ensure that there is a Ritz
value $\theta_{j}^{(N)}$ to within $\epsilon_{1}\|A\|$ of $\theta_{i}^{(N)}$
for which $\|z_{j}^{(N)}\|\geq0.5$. Based on the previous reasoning,
it holds that
\[
\min_{l}|\lambda_{l}-\theta_{i}^{(N)}|\leq\min_{l}|\lambda_{l}-\theta_{j}^{(N)}|+|\theta_{j}^{(N)}-\theta_{i}^{(N)}|\leq\left(2\sqrt{N}\epsilon_{2}+\epsilon_{1}\right)\|A\|,
\]
which implies the bound (\ref{eq:size}).
\end{proof}
In summary, if one can find perturbations $\Delta\widetilde{V}_{N-k+1}$
with small norm, say $\mathcal{O}(\sqrt{\epsilon})\|A\|$, such that
(\ref{eq:Grmodel}) holds and if the assumption from Theorems~\ref{thm:size_intervals}
regarding $\epsilon_{1}$ is satisfied for some $\epsilon_{1}=\mathcal{O}(\sqrt{\epsilon})$,
then the extended matrix $T_{N}$ must have eigenvalues in intervals
of width $\mathcal{O}(\sqrt{\epsilon})\|A\|$ around the eigenvalues
of $A$. Although the existence of a sufficiently small $\epsilon_{1}$
is difficult to establish theoretically, our experiments in Section~\ref{sec:Experiments}
indicate that $\epsilon_{1}$ is consistently small enough and has
negligible impact on the bound (\ref{eq:size}).

\subsection{The continuation process \protect\label{subsec:The-continuation-process}}

In this section, we present a construction that leads to (\ref{eq:Grmodel})
after $k$ iterations of the finite precision block Lanczos algorithm.
Specifically, we construct the perturbations $\Delta\widetilde{V}_{N-k+1}$
in a particular way, using information obtained from the finite precision
iterations. This construction is analogous to that introduced in \cite{Gr1989}
for the single-vector Lanczos algorithm, and will be referred to as
the \textit{continuation process}.

Let $W_{k}$ be an $n\times m$ matrix such that $W_{k}^{T}W_{k}=I_{m}$.
Suppose that the block vectors $v_{k-1}$, $v_{k}$, and the block
coefficients $\alpha_{k}$ and $\beta_{k}$ have already been computed,
for instance, after $k$ iterations of the finite precision block
Lanczos algorithm applied to $A$ and $v$. We now describe a procedure
for generating block vectors $q_{k+j}$, for $j=1,2,\dots,$ by orthogonalizing
$Aq_{k+j-1}$ against $W_{k}$, $q_{k+1}$, and the most recently
computed vector $q_{k+j-1}$. Consider the following construction
\begin{equation}
\begin{aligned}\tilde{q}_{k+1} & =Av_{k}-v_{k}\alpha_{k}-v_{k-1}\beta_{k}^{T}, & q_{k+1}\beta_{k+1} & =(I_{n}-P_{1})\tilde{q}_{k+1},\\
\tilde{q}_{k+2} & =Aq_{k+1}-v_{k}\beta_{k+1}^{T}, & q_{k+2}\beta_{k+2} & =(I_{n}-P_{2})\tilde{q}_{k+2},\\
\tilde{q}_{k+j} & =Aq_{k+j-1}-q_{k+j-2}\beta_{k+j-1}^{T}, & q_{k+j}\beta_{k+j} & =(I_{n}-P_{j})\tilde{q}_{k+j},\:j\geq3,
\end{aligned}
\label{eq:CP_I}
\end{equation}
where $P_{i}$ are orthogonal projectors
\begin{align*}
P_{1} & =W_{k}W_{k}^{T},\\
P_{2} & =P_{1}+q_{k+1}q_{k+1}^{T},\\
P_{j} & =P_{2}+q_{k+j-1}q_{k+j-1}^{T}.
\end{align*}
The columns of block vectors $q_{k+j}$, for $j\geq1$, are defined
as orthonormal bases of the column spaces of $(I_{n}-P_{j})\tilde{q}_{k+j}$.
In cases of rank deficiency, the number of columns in $q_{k+j}$ is
reduced accordingly, and the corresponding block $\beta_{k+j}$ becomes
rectangular with full row rank. Implementation details for computing
$q_{k+j}$ and $\beta_{k+j}$ are provided in Section~\ref{subsec:Implement_CP}. 

We now prove that the block vectors produced by the process (\ref{eq:CP_I})
are orthonormal. 
\begin{thm}
\label{thm:orthogonality}The set of columns of $W_{k}$ and $q_{k+1},\ldots,q_{k+j}$,
given by (\ref{eq:CP_I}), is orthonormal. Moreover, for $j\geq3$,
it holds that
\[
\beta_{k+j}=q_{k+j}^{T}Aq_{k+j-1}.
\]
\end{thm}

\begin{proof}
The orthogonality of the columns of $[W_{k},q_{k+1},q_{k+2},q_{k+3}]$
follows directly from the definition of the process. It is also easy
to see that
\[
\beta_{k+3}=q_{k+3}^{T}(I_{n}-P_{3})\tilde{q}_{k+3}=q_{k+3}^{T}Aq_{k+2}.
\]

We will prove the theorem by induction. Let the columns of $[W_{k},q_{k+1},\ldots,q_{k+j-1}]$
be orthonormal and let $\beta_{k+i}=q_{k+i}^{T}Aq_{k+i-1}$ hold for
$3\leq i\leq j-1$ for a given $j>3$. By the induction hypothesis
and the definition of $q_{k+j}$, the block vector $q_{k+j}$ is orthogonal
to $W_{k},q_{k+1}$ and $q_{k+j-1}$. For $1<i<j-1$ we obtain
\begin{align}
\textbf{\ensuremath{q_{k+i}^{T}}}q_{k+j}\beta_{k+j} & =q_{k+i}^{T}(I_{n}-P_{j})\tilde{q}_{k+j}=q_{k+i}^{T}\tilde{q}_{k+j}\nonumber \\
 & =q_{k+i}^{T}(Aq_{k+j-1}-q_{k+j-2}\beta_{k+j-1}^{T}).\label{eq:CP_pom}
\end{align}
For $i=j-2$, the right-hand side of (\ref{eq:CP_pom}) is zero by
the induction hypothesis applied to $\beta_{k+j-1}$. For $1<i<j-2$,
the right-hand side of (\ref{eq:CP_pom}) simplifies as 
\[
q_{k+i}^{T}Aq_{k+j-1}=\left(q_{k+j-1}^{T}(\tilde{q}_{k+i+1}+q_{k+i-1}\beta_{k+i}^{T})\right)^{T}=0.
\]
Since (\ref{eq:CP_pom}) has just been shown to be zero and $\beta_{k+j}$
has full row rank in all cases, it holds that $\textbf{\ensuremath{q_{k+i}^{T}}}q_{k+j}=0$.
This completes the induction step for orthogonality. Finally, we obtain
\[
\beta_{k+j}=q_{k+j}^{T}q_{k+j}\beta_{k+j}=q_{k+j}^{T}(I_{n}-P_{j})\tilde{q}_{k+j}=q_{k+j}^{T}\tilde{q}_{k+j}=q_{k+j}^{T}Aq_{k+j-1}.
\]
\end{proof}
The following theorem shows that the process defined in (\ref{eq:CP_I})
can also be interpreted as a perturbed three-term recurrence.
\begin{thm}
\label{thm:hj}The process defined in (\ref{eq:CP_I}) can be written
in following way
\begin{equation}
\begin{aligned}q_{k+1}\beta_{k+1} & =Av_{k}-v_{k}\alpha_{k}-v_{k-1}\beta_{k}^{T}-h_{k},\\
q_{k+2}\beta_{k+2} & =Aq_{k+1}-q_{k+1}\alpha_{k+1}-v_{k}\beta_{k+1}^{T}-h_{k+1},\\
q_{k+j}\beta_{k+j} & =Aq_{k+j-1}-q_{k+j-1}\alpha_{k+j-1}-q_{k+j-2}\beta_{k+j-1}^{T}-h_{k+j-1},\:j\geq3,
\end{aligned}
\label{eq:CP_II}
\end{equation}
where \textup{
\begin{equation}
\begin{aligned}\alpha_{k+1} & =q_{k+1}^{T}(Aq_{k+1}-v_{k}\beta_{k+1}^{T}),\\
\alpha_{k+j-1} & =q_{k+j-1}^{T}Aq_{k+j-1},
\end{aligned}
\label{eq:alphas}
\end{equation}
} and\textup{
\begin{align*}
h_{k} & =P_{1}\left(Av_{k}-v_{k}\alpha_{k}-v_{k-1}\beta_{k}^{T}\right),\\
h_{k+1} & =P_{1}\left(Aq_{k+1}-v_{k}\beta_{k+1}^{T}\right),\\
h_{k+j-1} & =P_{1}Aq_{k+j-1}+q_{k+1}\beta_{k+1}v_{k}^{T}q_{k+j-1},
\end{align*}
}with $P_{1}=W_{k}W_{k}^{T}$.
\end{thm}

\begin{proof}
The claim can be readily verified for $q_{k+1}$ and $q_{k+2}$. This
follows by substituting the definitions of the projectors $P_{1}$
and $P_{2}$, along with the coefficient $\alpha_{k+1}$ as given
in (\ref{eq:alphas}), into the process described in (\ref{eq:CP_I}).

We now focus on the case $j\geq3$. First, it holds that
\begin{align*}
q_{k+1}^{T}Aq_{k+j-1} & =(\tilde{q}_{k+2}+v_{k}\beta_{k+1}^{T})^{T}q_{k+j-1}\\
 & =(q_{k+2}\beta_{k+2}+P_{2}\tilde{q}_{k+2}+v_{k}\beta_{k+1}^{T})^{T}q_{k+j-1},
\end{align*}
and therefore
\begin{equation}
q_{k+1}^{T}Aq_{k+j-1}=\begin{cases}
\beta_{k+1}v_{k}^{T}q_{k+j-1}+\beta_{k+2}^{T}, & j=3,\\
\beta_{k+1}v_{k}^{T}q_{k+j-1}, & j>3.
\end{cases}\label{eq:showing_equality}
\end{equation}
Using (\ref{eq:showing_equality}) it can be shown that 
\begin{equation}
q_{k+1}(q_{k+1}^{T}Aq_{k+j-1}-q_{k+1}^{T}q_{k+j-2}\beta_{k+j-1}^{T})=q_{k+1}\beta_{k+1}v_{k}^{T}q_{k+j-1}.\label{eq:showing_equality2}
\end{equation}
Finally, by defining $\alpha_{k+j-1}$ as in (\ref{eq:alphas}) and
using (\ref{eq:showing_equality2}) together with Theorem~\ref{thm:orthogonality}
we can complete the proof for $j\geq3$.
\end{proof}
To summarize this section, we have introduced the continuation process,
which completes the $k$ iterations of the finite precision block
Lanczos algoritm using pertubed three-term recurrences. This process
involves only one free parameter: the matrix $W_{k}$. Since $W_{k}$
appears in the expressions for $h_{k+j}$, it directly influences
the size of the perturbations. The construction of a suitable $W_{k}$
is the focus of the next section.

%\subsection{Properties of $W_{k}$}
\subsection{Properties of \texorpdfstring{$W_{k}$}{Wk}}
\label{subsec:Construction-of-W}

In the previous section, we defined the continuation process, which
extends the finite precision block Lanczos computations beyond $k$
iterations and ultimately leads to (\ref{eq:Grmodel}). To apply Theorem~\ref{thm:size_intervals}
and bound the spread of the eigenvalues of $T_{N}$ around those of
$A$, the perturbations introduced by the continuation process must
be sufficiently small. Crucially, the only parameter available to
control the size of these perturbations is the matrix $W_{k}$. As
such, choosing an appropriate $W_{k}$ is essential for ensuring the
effectiveness of the continuation process, and poses a key challenge
in the overall construction. This issue is the focus of the current
section.

As in the single-vector case, $W_{k}$ is obtained from the QR factorization
of a selected subset of $m$ Ritz vectors, denoted $Z_{m}^{(k)}$.
While no general theoretical criterion exists for selecting $Z_{m}^{(k)}$,
we can derive certain requirements on $W_{k}$ by analyzing the resulting
perturbations. In the following theorem, we present an alternative
formulation of $h_{k+j}$, $j=0,1,\ldots$, which highlights the terms
that most significantly influence the size of the perturbations. 
\begin{thm}
\label{thm:matrix_W} Let $T_{k},V_{k},v_{k+1},\beta_{k+1}^{FP}$
be the quantities obtained after $k$ iterations of the finite precision
block Lanczos algorithm applied to $A$ with an initial vector $v$,
satisfying \textup{(\ref{eq:perturbedBLanczos})}, \textup{(\ref{eq:FPperturbations}),
(\ref{eq:FPproperties-localOG})} and \textup{(\ref{eq:FPbeta})}.
Let \textup{(\ref{def:spectral_decomp-T}) }denote the spectral decomposition
of\textup{ $T_{k}$} and define $r_{k}^{T}=[v_{k}^{T}V_{k-1},0_{p}]$.
Let $Z_{m}^{(k)}$ be a selected subset of $m$ linearly independent
Ritz vectors, with QR factorization $Z_{m}^{(k)}=W_{k}R_{k}$. Then,
considering the continuation process defined by\textup{ (\ref{eq:CP_II}),}
it holds that
\begin{equation}
\begin{aligned}h_{k} & =W_{k}W_{k}^{T}v_{k+1}\beta_{k+1}^{FP}+\Delta_{0}^{(k)},\\
h_{k+1} & =-W_{k}R_{k}^{-T}S_{m}^{(k)T}r_{k}\beta_{k+1}^{T}+\Delta_{1}^{(k)},\\
h_{k+j} & =q_{k+1}\beta_{k+1}v_{k}^{T}q_{k+j}+\Delta_{j}^{(k)},\quad j\geq2,
\end{aligned}
\label{eq:norms_perturbations}
\end{equation}
 where
\[
\begin{aligned}\|\Delta_{0}^{(k)}\| & \leq\mathcal{O}(\epsilon)\|A\|,\\
\|\Delta_{j}^{(k)}\| & \leq(1+\rho_{k})\|h_{k}\|+\left(1+\rho_{k}\right)\mathcal{O}(\epsilon)\|A\|,\quad j\geq1,
\end{aligned}
\]
and\textup{ $\rho_{k}=\|R_{k}^{-1}\|$.}
\end{thm}

\begin{proof}
The first perturbation term $h_{k}$ in (\ref{eq:CP_II}) can be written
as
\[
h_{k}=W_{k}W_{k}^{T}\left(Av_{k}-v_{k}\alpha_{k}-v_{k-1}\beta_{k}^{T}\right)=W_{k}W_{k}^{T}v_{k+1}\beta_{k+1}^{FP}+\Delta_{0}^{(k)},
\]
where $\Delta_{0}^{(k)}=W_{k}W_{k}^{T}\Delta v_{k}$. Using (\ref{eq:FPperturbations}),
we obtain the bound $\|\Delta_{0}^{(k)}\|\leq\mathcal{O}(\epsilon)\|A\|$. 

Let us now express and bound the perturbation terms $h_{k+j}$ for
$j\geq1$. We begin by examining the expressions $W_{k}W_{k}^{T}Aq_{k+j}$. 

Multiplying (\ref{eq:perturbedBLanczos}) from the right by $S_{m}^{(k)}R_{k}^{-1}$,
we arrive at
\begin{equation}
AW_{k}=W_{k}R_{k}\Theta_{m}^{(k)}R_{k}^{-1}+v_{k+1}\beta_{k+1}^{FP}e_{k}^{T}S_{m}^{(k)}R_{k}^{-1}+E_{1}^{(k)},\label{eq:AW}
\end{equation}
where 
\[
E_{1}^{(k)}=\Delta V_{k}S_{m}^{(k)}R_{k}^{-1},
\]
and $\Theta_{m}^{(k)}$ is the diagonal matrix of Ritz values corresponding
to the selected Ritz vectors stored in $Z_{m}^{(k)}$. Using (\ref{eq:FPperturbations}),
the size of $E_{1}^{(k)}$ can be bounded by 
\begin{equation}
\|E_{1}^{(k)}\|\leq\rho_{k}\,\mathcal{O}(\epsilon)\|A\|,\label{eq:estimE1}
\end{equation}
where $\rho_{k}=\|R_{k}^{-1}\|$. 

Now, let us focus on expressing the term $e_{k}^{T}S_{m}^{(k)}R_{k}^{-1}$,
which appears in the middle term on the right-hand side of (\ref{eq:AW}).
First realize that
\begin{eqnarray*}
v_{k}^{T}W_{k} & = & v_{k}^{T}Z_{m}^{(k)}R_{k}^{-1}\\
 & = & v_{k}^{T}\left[V_{k-1},v_{k}\right]S_{m}^{(k)}R_{k}^{-1}\\
 & = & r_{k}^{T}S_{m}^{(k)}R_{k}^{-1}+v_{k}^{T}v_{k}e_{k}^{T}S_{m}^{(k)}R_{k}^{-1}.
\end{eqnarray*}
Writing $v_{k}^{T}v_{k}=I_{p}+F_{k}$, we obtain 
\begin{equation}
e_{k}^{T}S_{m}^{(k)}R_{k}^{-1}=v_{k}^{T}W_{k}-r_{k}^{T}S_{m}^{(k)}R_{k}^{-1}-F_{k}e_{k}^{T}S_{m}^{(k)}R_{k}^{-1}.\label{eq:help12}
\end{equation}
Note that using (\ref{eq:FPproperties-normbLanczvecs}), we have 
\begin{equation}
\|F_{k}\|\leq\mathcal{O}(\epsilon),\label{eq:help2}
\end{equation}

Using (\ref{eq:help12}), the relation (\ref{eq:AW}) can be written
in the form
\begin{eqnarray}
AW_{k} & = & W_{k}R_{k}\Theta_{m}^{(k)}R_{k}^{-1}\label{eq:AW_approx}\\
 &  & +\,v_{k+1}\beta_{k+1}^{FP}\left(v_{k}^{T}W_{k}-r_{k}^{T}S_{m}^{(k)}R_{k}^{-1}\right)+\widetilde{F}_{k}+E_{1}^{(k)},\nonumber 
\end{eqnarray}
where $\widetilde{F}_{k}=-v_{k+1}\beta_{k+1}^{FP}F_{k}e_{k}^{T}S_{m}^{(k)}R_{k}^{-1}$.
Using (\ref{eq:FPproperties-normbLanczvecs}), (\ref{eq:FPbeta})
and (\ref{eq:help2}) we obtain 
\begin{equation}
\|\widetilde{F}_{k}\|\leq\rho_{k}\,\mathcal{O}(\epsilon)\|A\|.\label{eq:help3}
\end{equation}

Further, denoting $E_{j+1}^{(k)}=q_{k+j}^{T}(h_{k}-\Delta v_{k})$
for $j\geq1$, we get
\begin{equation}
q_{k+j}^{T}v_{k+1}\beta_{k+1}^{FP}=q_{k+j}^{T}\left(q_{k+1}\beta_{k+1}+h_{k}-\Delta v_{k}\right)=\begin{cases}
\beta_{k+1}+E_{2}^{(k)}, & j=1,\\
E_{j+1}^{(k)}, & j>1,
\end{cases}\label{eq:definingW}
\end{equation}
and the size of $E_{j+1}^{(k)}$ can be bounded, using (\ref{eq:FPperturbations}),
as follows
\begin{equation}
\|E_{j+1}^{(k)}\|\leq\|h_{k}\|+\mathcal{O}(\epsilon)\|A\|,\quad j\geq1.\label{eq:estimEj+1}
\end{equation}

Combining the algebraic expressions (\ref{eq:AW_approx}) and (\ref{eq:definingW}),
we obtain 
\begin{eqnarray}
W_{k}\left(W_{k}^{T}Aq_{k+1}\right) & = & W_{k}\left(q_{k+1}^{T}AW_{k}\right)^{T}\nonumber \\
 & = & W_{k}\left(v_{k}^{T}W_{k}-r_{k}^{T}S_{m}^{(k)}R_{k}^{-1}\right)^{T}\left(q_{k+1}^{T}v_{k+1}\beta_{k+1}^{FP}\right)^{T}\nonumber \\
 &  & +\,W_{k}(\widetilde{F}_{k}+E_{1}^{(k)})^{T}q_{k+1}\nonumber \\
 & = & W_{k}\left(v_{k}^{T}W_{k}-r_{k}^{T}S_{m}^{(k)}R_{k}^{-1}\right)^{T}\beta_{k+1}^{T}+\Delta_{1}^{(k)},\label{eq:help3-2}
\end{eqnarray}
where 
\[
\Delta_{1}^{(k)}=W_{k}\left(v_{k}^{T}W_{k}-r_{k}^{T}S_{m}^{(k)}R_{k}^{-1}\right)^{T}E_{2}^{(k)T}+W_{k}(\widetilde{F}_{k}+E_{1}^{(k)})^{T}q_{k+1}.
\]

To complete the proof, we now express the perturbation terms $h_{k+j}$
using the previous results and Theorem~\ref{thm:hj}. For $j=1$,
we obtain, using (\ref{eq:CP_II}) and (\ref{eq:help3-2}),
\[
h_{k+1}=W_{k}W_{k}^{T}\left(Aq_{k+1}-v_{k}\beta_{k+1}^{T}\right)=-W_{k}R_{k}^{-T}\left(S_{m}^{(k)}\right)^{T}r_{k}\beta_{k+1}^{T}+\Delta_{1}^{(k)},
\]
and for $j>1$, we get, using (\ref{eq:CP_II}), 
\begin{align*}
h_{k+j} & =W_{k}W_{k}^{T}\left(Aq_{k+j}-q_{k+j-1}\beta_{k+j}^{T}\right)+q_{k+1}\beta_{k+1}v_{k}^{T}q_{k+j}\\
 & =q_{k+1}\beta_{k+1}v_{k}^{T}q_{k+j}+\Delta_{j}^{(k)},
\end{align*}
where we denoted 
\[
\Delta_{j}^{(k)}=W_{k}\left(W_{k}^{T}Aq_{k+j}\right).
\]
The term $\Delta_{j}^{(k)}$ can be expressed, using (\ref{eq:AW_approx})
and (\ref{eq:definingW}), in the form
\begin{eqnarray*}
\Delta_{j}^{(k)} & = & W_{k}\left(v_{k+1}\beta_{k+1}^{FP}\left(v_{k}^{T}W_{k}-r_{k}^{T}S_{m}^{(k)}R_{k}^{-1}\right)+\widetilde{F}_{k}+E_{1}^{(k)}\right)^{T}q_{k+j}\\
 & = & W_{k}\left(v_{k}^{T}W_{k}-r_{k}^{T}S_{m}^{(k)}R_{k}^{-1}\right)^{T}\left(E_{j+1}^{(k)}\right)^{T}+W_{k}(\widetilde{F}_{k}+E_{1}^{(k)})^{T}q_{k+j}.
\end{eqnarray*}

Finally, from (\ref{eq:estimE1}), (\ref{eq:estimEj+1}) and (\ref{eq:help3})
it follows that
\[
\|\Delta_{j}^{(k)}\|\leq(1+\rho_{k})\|h_{k}\|+\left(1+\rho_{k}\right)\mathcal{O}(\epsilon)\|A\|,\quad j\geq1,
\]
which finished the proof.
\end{proof}
If the columns of $Z_{m}^{(k)}$ are sufficiently linearly independent,
then $\rho_{k}\leq\mathcal{O}(1)$. In this case, Theorem~\ref{thm:matrix_W}
implies that the size of $h_{k}$ and $h_{k+1}$ depends primarily
on the terms 
\begin{equation}
\begin{gathered}\|W_{k}^{T}v_{k+1}\beta_{k+1}^{FP}\|\quad\mbox{and\ensuremath{\quad}\ensuremath{\|\beta_{k+1}r_{k}^{T}S_{m}^{(k)}R_{k}^{-1}\|}.}\end{gathered}
\label{eq:perturbation_norms1}
\end{equation}
To keep the size of the remaining perturbations $h_{k+j}$ for
$j\geq2$ small, we can present a sufficient condition based on Theorem~\ref{thm:matrix_W}.
Since the matrix $W_{k}$ has orthonormal columns, 
it holds that
\begin{eqnarray}
\left\Vert q_{k+j}^{T}v_{k}\beta_{k+1}^{T}\right\Vert  & = & \left\Vert q_{k+j}^{T}\left(I-W_{k}W_{k}^{T}\right)v_{k}\beta_{k+1}^{T}\right\Vert \nonumber \\
 & \leq & \left\Vert (I-W_{k}W_{k}^{T})v_{k}{\normalcolor {\color{black}\beta_{k+1}^{T}}}\right\Vert \label{eq:perturbation_norms2}
\end{eqnarray}
Therefore,  the perturbations $h_{k+j}$, $j\geq2$, remain small if the term 
\eqref{eq:perturbation_norms2} is sufficiently small.

Let us now discuss what properties of $Z_{m}^{(k)}$ could ensure
that the terms in (\ref{eq:perturbation_norms1}) and (\ref{eq:perturbation_norms2})
remain small. The first term in (\ref{eq:perturbation_norms1}) is
small if $v_{k+1}$ is nearly orthogonal to the selected Ritz vectors
stored in $Z_{m}^{(k)}$. For the term (\ref{eq:perturbation_norms2})
to be small, the columns of the block vector $v_{k}$ should lie approximately
in the space generated by the columns of $Z_{m}^{(k)}$. These two
sufficient conditions will form the basis of our selection criterion.

The interpretation of the second term in (\ref{eq:perturbation_norms1})
is nontrivial even in the single-vector case. Nevertheless, in that
setting, Greenbaum established an upper bound for this term based
on results of Paige, as summarized in \cite[Lemma (Paige),  p. 32]{Gr1989}.
In the block case, however, the problem of bounding this term remains
open.

Generally, there is no theory ensuring that there is a way to compose
$Z_{m}^{(k)}$ yielding small terms in (\ref{eq:perturbation_norms1})
and (\ref{eq:perturbation_norms2}). However, based on the experiments
presented in the following section, it seems that we can use a selection
criterion analogous to the single-vector case.

\section{Experiments \protect\label{sec:Experiments}}

In the previous section, we defined the continuation process (\ref{eq:CP_II}),
which leads to the construction of (\ref{eq:Grmodel}), and we stated
Theorem~\ref{thm:size_intervals}, providing a bound on the distance
between the eigenvalues of $T_{N}$ and those of $A$. When the perturbations
$\Delta\widetilde{V}_{N-k+1}$ are sufficiently small, the eigenvalues
of $T_{N}$ cluster around those of~$A$. The size of these perturbations
depends on the free parameter $W_{k}$ in the continuation process.
Although no general theory guarantees small perturbations for arbitrary
choices of $W_{k}$, we showed in the previous section that ensuring
small values for the quantities in (\ref{eq:perturbation_norms1})
and (\ref{eq:perturbation_norms2}) is sufficient. In this section,
we perform experiments with a criterion for selecting $W_{k}$ inspired
by the single-vector case, and check the sizes of (\ref{eq:perturbation_norms1})
and (\ref{eq:perturbation_norms2}) numerically. We also examine how
closely the eigenvalues of $T_{N}$ cluster around those of $A$.

\subsection{Construction of \texorpdfstring{$W_{k}$}{Wk}}
\label{subsec:ConstructionW}

Let $T_{k}$ be the block tridiagonal matrix obtained after $k$ iterations
of the finite precision block Lanczos algorithm applied to $A$ and
$v$. In the context of the continuation process (\ref{eq:CP_II}),
our goal is to define the matrix $W_{k}$ so that the quantities in
(\ref{eq:perturbation_norms1}) and (\ref{eq:perturbation_norms2})
remain small, which in turn leads to small perturbations $h_{k+j}$,
$j=0,1,\ldots$. Following the analogy with the single-vector case,
it is advantageous to construct $W_{k}$ from a carefully chosen subset
of Ritz vectors $Z_{m}^{(k)}$. Specifically, $W_{k}$ is taken as
the $Q$-factor from the QR factorization of $Z_{m}^{(k)}$. Since
there is no theoretical framework for choosing this subset in the
block case, we employ a heuristic criterion inspired by the single-vector
setting. We then verify the resulting sizes of terms in (\ref{eq:perturbation_norms1})
and (\ref{eq:perturbation_norms2}) through numerical experiments. 

From \cite{Pa1980}, it is known that in the single-vector Lanczos
algorithm, a well-separated cluster of Ritz values cannot be associated
with Ritz vectors that all have small norms. Building on this result,
Greenbaum introduced in \cite{Gr1989} the concept of cluster vectors,
representative vectors associated with clusters of Ritz values. When
selecting $Z_{m}^{(k)}$, Greenbaum used the unconverged Ritz vectors
for well-separated Ritz values, or unconverged cluster vectors for
well-separated clusters. Since there is currently no theoretical framework
for generalizing the notion of cluster vectors to the block case,
we adopt a simplified heuristic criterion for selecting $Z_{m}^{(k)}$,
\begin{equation}
\delta_{k,i}>\mu\|A\|,\label{eq:unconverged}
\end{equation}
where $\mu>0$ is a small constant.

Before turning to the experiments, we comment on the term (\ref{eq:perturbation_norms2}).
Let $T_{k}$ have the eigendecomposition~(\ref{def:spectral_decomp-T}),
and define the Ritz vectors of $T_{k}$ as $Z_{k}=V_{k}S_{k}$. Let
$Z_{m}^{(k)}=V_{k}S_{m}^{(k)}$ denote the subset of Ritz vectors
selected according to the criterion (\ref{eq:unconverged}) for a
given constant $\mu>0$. From the definition of the Ritz vectors,
the block vector $v_{k}{\color{black}\beta_{k+1}^{T}}$ can be expressed
as
\[
v_{k}{\normalcolor {\color{black}\beta_{k+1}^{T}}}=\sum_{i=1}^{kp}z_{i}^{(k)}\sigma_{p,i}^{(k)T}{\normalcolor {\color{black}\beta_{k+1}^{T}}},
\]
where $\sigma_{p,i}^{(k)}$ are defined as in (\ref{eq:RitzVecs}).
For the unselected Ritz vectors it holds that
\[
\|z_{i}^{(k)}\sigma_{p,i}^{(k)T}{\normalcolor {\color{black}\beta_{k+1}^{T}}}\|\leq\mu\|A\|\|z_{i}^{(k)}\|.
\]
In our numerical experiments, the norms of the Ritz vectors are never
significantly greater than one, and thus we may write
\[
v_{k}{\normalcolor {\color{black}\beta_{k+1}^{T}}}=\sum_{j=1}^{m}z_{i_{j}}^{(k)}\sigma_{p,i_{j}}^{(k)T}{\normalcolor {\color{black}\beta_{k+1}^{T}}}+\mathcal{O}(\mu)\|A\|,
\]
where $i_{1},\ldots,i_{m}$ are the indices of the selected Ritz vectors
that form the matrix $Z_{m}^{(k)}$. Therefore, if the Ritz vectors
are selected according to (\ref{eq:unconverged}), the term (\ref{eq:perturbation_norms2})
is of the order $\mathcal{O}(\mu)\|A\|$.

In the following experiment, our aim is to construct a matrix for
which improper clusters of Ritz values also appear during finite precision
computations; see Section~\ref{subsec:Clusters-not-approximating}.
Specifically, we choose $B$ as $\texttt{strakos48(0.001,1)}$ or
$\texttt{strakos48(0.1,100)}$, and use the initial vector $y=\texttt{randn(n,p)}$.
We first apply the single-vector Lanczos algorithm with double reorthogonalization
(to simulate exact arithmetic) to $B$ and $y$, producing a tridiagonal
matrix $\widetilde{T}_{s}$ that has the same eigenvalues as $B$,
where $s$ is the degree of $y$ with respect to $B$. We define the
test matrix as
\[
A=U\left(\widetilde{T}_{s}\otimes(I_{p}+\omega E)\right)U^{T},
\]
where $E\in\mathbb{R}^{p\times p}$ is a random matrix, $\omega=10^{-12}$
is a small perturbation parameter, and $U$ is a random orthonormal
matrix. The initial vector $v$ is defined as
\[
v=U\left(e_{1}\otimes I_{p}\right),
\]
where $e_{1}\in\mathbb{R}^{s}$ is the first column of the identity
matrix $I_{s}$. We refer to such a matrix $A$ as $\texttt{strakos48(0.001,1)}_{\otimes}$
or $\texttt{strakos48(0.1,100)}_{\otimes}$, respectively.

We now apply the block Lanczos algorithm to $\texttt{strakos48(0.001,1)}_{\otimes}$
and $\texttt{strakos48(0.1,100)}_{\otimes}$ and the initial vector
$v$ in finite precision arithmetic. In  Figure~\ref{img:experiment3},
we plot the quantities (\ref{eq:perturbation_norms1}) and (\ref{eq:perturbation_norms2})
for $p=2$, using the tolerance $\mu=10^{-5}$ in the selection criterion
(\ref{eq:unconverged}). The left part of the figure corresponds to
$\texttt{strakos48(0.001,1)}_{\otimes}$ and the right part to $\texttt{strakos48(0.1,100)}_{\otimes}$.
These quantities are compared with the threshold $\mu\|A\|$ (dashed
line). As observed, the plotted terms stay within the order of $\mu\|A\|$
for all iterations $k$. Decreasing the tolerance $\mu$ may lead
to issues with the selected Ritz vectors associated with Ritz values
in proper clusters. In such cases, the quantities
in~(\ref{eq:perturbation_norms1}) can significantly exceed $\mu\|A\|$.
Similar behavior can also be observed for $p>2$, though it was necessary
to increase the tolerance, possibly due to the unclear influence of
the parameter $p$. From our numerical experiments related to Section~\ref{subsec:Clusters-not-approximating},
we know that the Ritz vectors corresponding to improper clusters are
nearly orthogonal to both $v_{k+1}$ and to each other. Therefore,
they are not expected to significantly affect the quantities studied
in this experiment.

\begin{figure}
\includegraphics[width=0.48\textwidth]{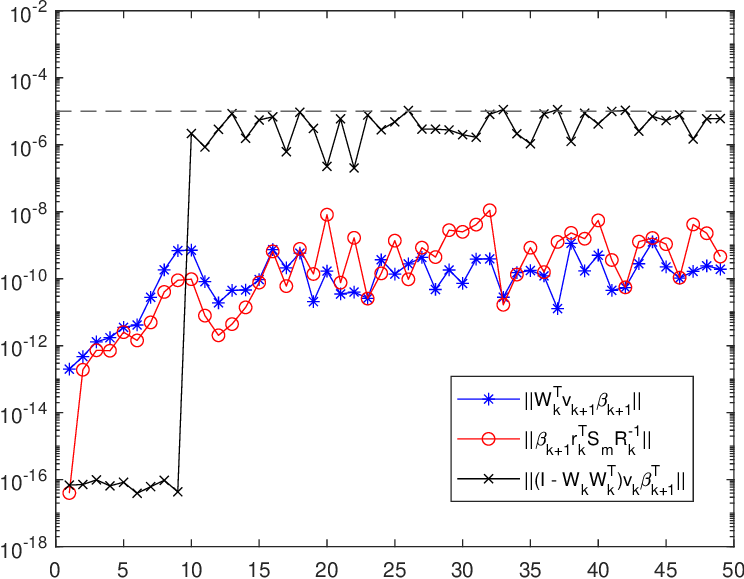}\includegraphics[width=0.48\textwidth]{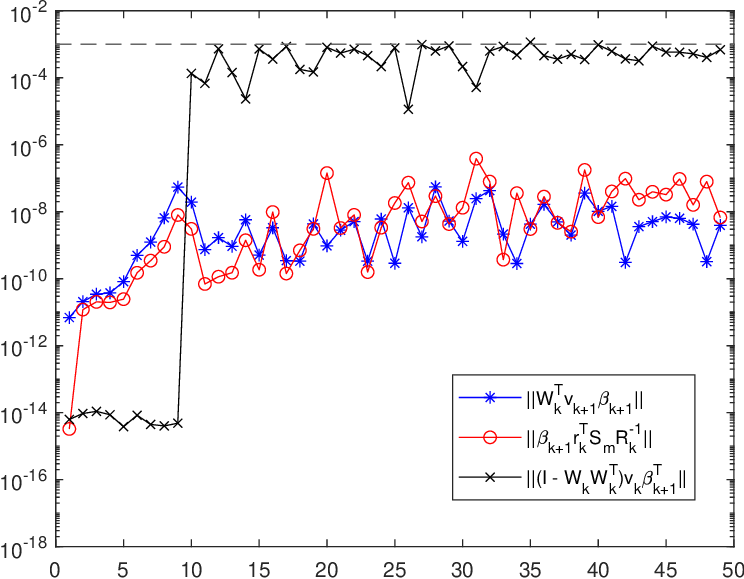}\caption{The terms (\ref{eq:perturbation_norms1}) and (\ref{eq:perturbation_norms2})
for $\texttt{strakos48(0.001,1)}_{\otimes}$ (left) and $\texttt{strakos48(0.1,100)}_{\otimes}$
(right), computed with the selection criterion (\ref{eq:unconverged})
for $\mu=10^{-5}$. The threshold $\mu\|A\|$ is indicated by a dashed
line. \protect\label{img:experiment3}}
\end{figure}

In the experiments, we found a tolerance $\mu$ for which the quantities
(\ref{eq:perturbation_norms1}) and (\ref{eq:perturbation_norms2})
are $\mathcal{O}(\mu\|A\|)$ for all inputs used in this paper. 

\subsection{The implementation of the continuation process \protect\label{subsec:Implement_CP}}

After $k$ iterations of the finite precision block Lanczos algorithm,
the matrix $T_{k}$ is obtained. A subset of the Ritz vectors is then
selected according to the criterion (\ref{eq:unconverged}), forming
the matrix $Z_{m}^{(k)}$. The matrix $W_{k}$ is subsequently computed
using the QR factorization of $Z_{m}^{(k)}$. We now describe how
the continuation process (\ref{eq:CP_II}), which produces $T_{N}$,
is implemented in MATLAB. It is important to note that the process
described by (\ref{eq:CP_II}) is a mathematical construction intended
to be carried out in exact arithmetic. Rather than using variable-precision
arithmetic to mimic exact computations, we employ numerical techniques
designed to ensure that the computed results closely approximate the
exact quantities.

To improve clarity, the recurrences in (\ref{eq:CP_II}) can be rewritten
in the form
\begin{align*}
q_{k+j}\beta_{k+j} & =\widetilde{w}_{k+j}-h_{k+j-1},\:j\geq1.
\end{align*}
The computation of $\widetilde{w}_{k+j}$, which represents the three-term
recurrence component of the continuation process, is implemented in
the same way as the three-term recurrence in Algorithm~\ref{alg:bLanczos}.
Each $\widetilde{w}_{k+j}$ is then reorthogonalized twice against
$W_{k}$ and, if applicable, against all previously computed block
vectors $q_{k+1},\ldots,q_{k+j-1}$. The resulting block vector is
denoted $w_{k+j}$. As mentioned above, an orthonormal basis of the
column space of $w_{k+j}$ must now be extracted. To achieve this,
we compute the economy-size singular value decomposition
\[
w_{k+j}=USV^{T}
\]
and discard singular values smaller than the tolerance $10^{-12}$.
This yields the truncated matrices $S_{t}$, $U_{t}$ and $V_{t}^{T}$,
and we define 
\[
q_{k+j}=U_{t},\quad\beta_{k+j}=S_{t}V_{t}^{T}.
\]
Finally, the sizes of the perturbations $h_{k+j-1}$ are computed
as 
\[
\|q_{k+j}\beta_{k+j}-\widetilde{w}_{k+j}\|.
\]
Note that in this implementation, the blocks $\beta_{k+j}$ are not
upper triangular; however, this does not pose any issues for our purposes.

\subsection{The matrix \texorpdfstring{$T_{N}$}{TN}}

In the previous section, we described how the matrix $T_{N}$ is obtained
using the continuation process. We now experimentally examine the
spread of its eigenvalues relative to those of $A$. In Section~\ref{subsec:ConstructionW},
we empirically determined that a tolerance of $\mu=10^{-5}$ in the
selection criterion (\ref{eq:unconverged}) ensures that the quantities
in (\ref{eq:perturbation_norms1}) and (\ref{eq:perturbation_norms2})
remain on the order of $\mu\|A\|$. Nevertheless, as we will now see,
the actual spread of the eigenvalues of $T_{N}$ around those of $A$
is typically a few orders of magnitude smaller than the bound $\epsilon_{2}\|A\|$
given in Theorem~\ref{thm:size_intervals}. 

In the first experiment we consider the matrix $A=\texttt{\ensuremath{\texttt{strakos48(0.001,1)}_{\otimes}}}$
with the same initial vector as in Section~\ref{subsec:ConstructionW}
and with $k=24$. The tolerance used in the selection criterion (\ref{eq:unconverged})
is set to $\mu=\sqrt{knp\epsilon}\approx10^{-6}$. The left part of
Figure~\ref{img:experiment4} shows the magnitudes of the perturbations
$h_{k+j}$ for $j=0,1,\dots,33$, compared with the threshold $\mu\|A\|$
(dashed line). The matrix $T_{N}$ was in this case of size $114\times114$.
The right part of the figure presents two plots: in the upper part,
the sizes of the intervals around the eigenvalues of $A$, normalized
by $\sqrt{\epsilon}\|A\|$, on a logarithmic scale;
and in the lower part, the number of eigenvalues of $T_{N}$ contained
in each interval. The largest interval sizes were on the order of
$10^{-12}$. 
\begin{figure}
\includegraphics[width=0.48\textwidth]{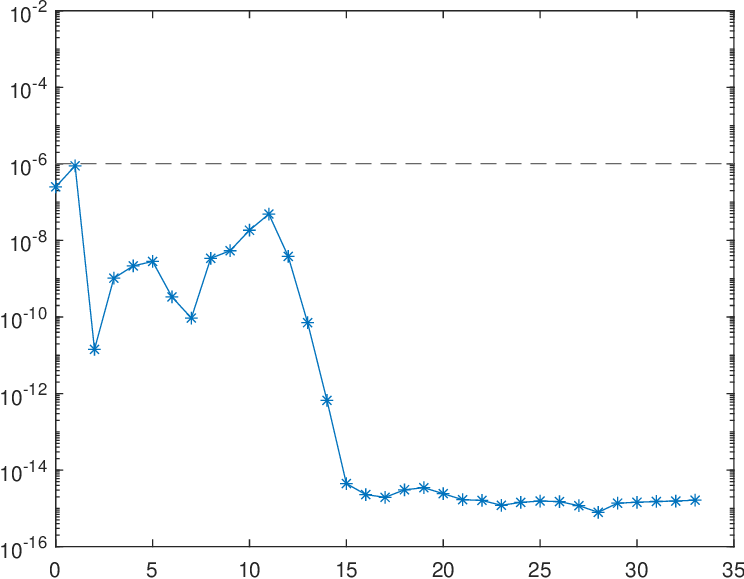}\,\,\includegraphics[width=0.48\textwidth]{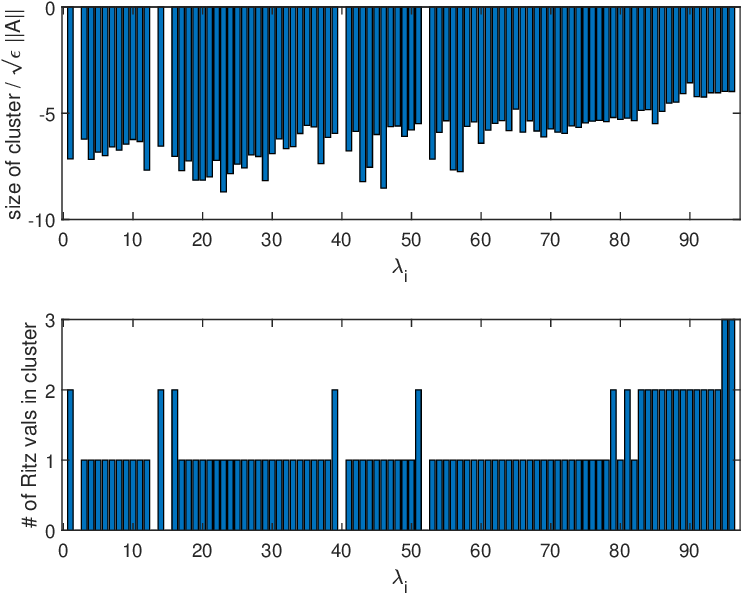}\caption{Experiment for the matrix $\texttt{strakos48(0.001,1)}_{\otimes}$
for $\mu=\sqrt{knp\epsilon}\approx10^{-6}$. Left: Norms of $h_{k+j}$
from the continuation process. Right top: Widths of the intervals
around the eigenvalues of $A$ that contain the eigenvalues of $T_{N}$,
normalized by $\sqrt{\epsilon}\|A\|$, on a logarithmic scale. Right bottom:
Number of eigenvalues of $T_{N}$ contained in each interval. \protect\label{img:experiment4}}
\end{figure}
In the second experiment, we use the matrix $A=\texttt{\ensuremath{\texttt{strakos48(0.1,100)}_{\otimes}}}$
also with the same initial vector as in Section~\ref{subsec:ConstructionW}
and with $k=24$. In this case we used a tolerance $\mu=10^{-5}$.
Figure~\ref{img:experiment5} displays the same quantities as in
the previous case. Here, the continuation process required $37$ iterations,
i.e., $j=0,\dots,36$, and produced a matrix $T_{N}$ of size $119\times119$.
The maximum interval size observed was on the order of $10^{-9}$.

In our experiments, the assumption (\ref{eq:assumtion}) in Theorem~\ref{thm:size_intervals}
was always satisfied, with the value of $\epsilon_{1}$ being at least
several orders of magnitude smaller than $\epsilon_{2}$. 

\begin{figure}
\includegraphics[width=0.48\textwidth]{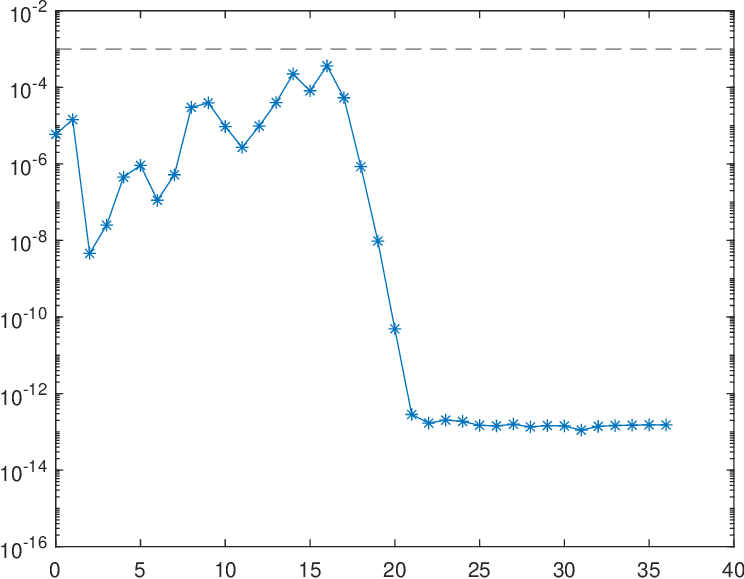}\,\,\includegraphics[width=0.48\textwidth]{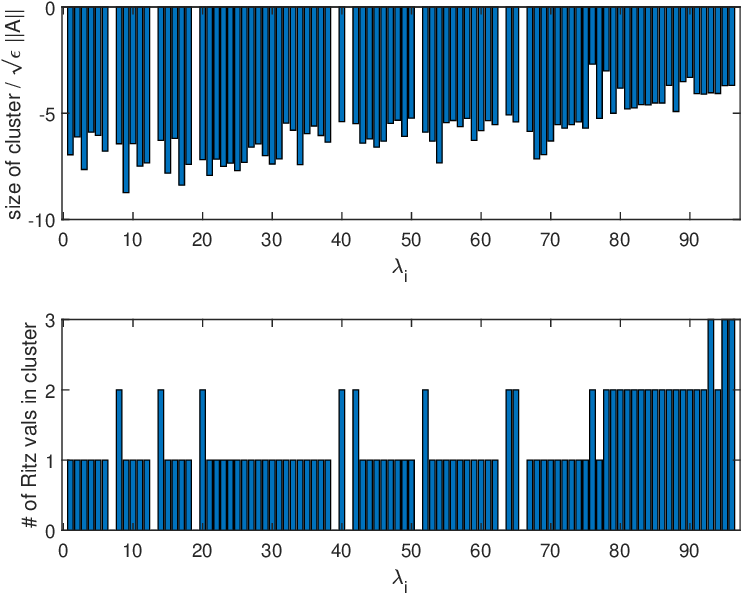}\caption{Experiment for the matrix $\texttt{strakos48(0.1,100)}_{\otimes}$
for $\mu=10^{-5}$. Left: Norms of $h_{k+j}$ from the continuation
process. Right top: Widths of the intervals around the eigenvalues
of $A$ that contain the eigenvalues of $T_{N}$,
normalized by $\sqrt{\epsilon}\|A\|$, on a logarithmic scale. Right bottom:
Number of eigenvalues of $T_{N}$ contained in each interval. \protect\label{img:experiment5}}
\end{figure}

\section{Conclusions}

The block Lanczos algorithm uses block operations that exploit modern
hardware and operates on a richer Krylov subspace, often yielding
faster convergence of the Ritz values to the eigenvalues than its
single-vector counterpart. However, unlike the single-vector case,
its behavior in finite precision arithmetic remains poorly understood.

Our goal was to extend the results introduced by Greenbaum \cite{Gr1989}
to the block setting. We reproduced the key experiment of Greenbaum
and Strakoš~\cite{GrSt1992} in the block setting. This experiment
indicates that the finite precision block Lanczos algorithm could
behave similarly as the exact block Lanczos algorithm applied to a
larger matrix whose eigenvalues are close to those of $A$. This observation
support the idea of backward-like stability of the block Lanczos algorithm,
analogous to that known for the single-vector algorithm \cite{Gr1989}.

In this paper, we generalized Greenbaum\textquoteright s continuation
process to the block setting. After performing $k$ finite precision
block Lanczos iterations, we continue the recurrences with carefully
designed perturbations so that the process terminates with $\beta_{N+1}=0$,
yielding a final block tridiagonal matrix $T_{N}$. Under an additional
assumption, Theorem~\ref{thm:size_intervals} shows that if the perturbations
are sufficiently small, the eigenvalues of $T_{N}$ cluster tightly
around those of $A$. A key open question is how to select the free
matrix parameter $W_{k}$ so that the designed perturbations indeed
remain small during the continuation process.
In the single-vector setting, Greenbaum leveraged Paige\textquoteright s
analysis \cite{Pa1980} to justify the construction of $W_{k}$. However,
to the best of our knowledge, Paige\textquoteright s results have
not been generalized to the block case. Therefore, we proposed an
empirical strategy: construct $W_{k}$ from a subset of Ritz vectors
that satisfy some sufficient conditions, using a simplified selection
criterion inspired by the single-vector case. We found parameters
such that the sufficient conditions were fulfilled with a tolerance
of order $10^{-5}$, while the observed spread of the eigenvalues
of $T_{N}$ around the eigenvalues of $A$ was typically $\mathcal{O}(\sqrt{\epsilon})\|A\|$,
and often even smaller. 

Our findings suggest that with an appropriate $W_{k}$, finite precision
block Lanczos computations can be viewed as the results of the exact
block Lanczos algorithm applied to a larger matrix. The eigenvalues
of this larger matrix lie in intervals of size $\mathcal{O}(\sqrt{\epsilon})\|A\|$
around the eigenvalues of $A$. A rigorous justification of this interpretation
would require block analogues of Paige\textquoteright s classical
results. At the same time, we believe that the results presented in
this paper provide a motivation for further analysis of the finite
precision behavior of the block Lanczos algorithm. Both our theoretical
developments and numerical experiments highlight which properties
are likely to extend to the block setting. A natural starting point
for a deeper analysis is a better understanding of how Ritz values
interlace for block tridiagonal matrices. Although we have formulated
and numerically supported a conjecture in this direction, establishing
a complete proof remains an interesting challenge for future research.

%\bibliography{references}

\end{document}